\documentclass[11pt]{amsart}

\usepackage{todonotes}

\usepackage[margin=.9in]{geometry}
\usepackage{booktabs}
\usepackage{graphicx, amsthm}
\usepackage{amssymb,amsthm,amsfonts,amsmath,hyperref,mathtools}
\usepackage{xcolor}
\definecolor{darkgreen}{rgb}{0,0.30,0} 
\usepackage{float}

\usepackage{multirow}
\usepackage{array}   
\newcolumntype{L}{>{$}l<{$}} 
\newcolumntype{C}{>{$}c<{$}} 

\usepackage{enumerate}
\hypersetup{
    colorlinks=true,
    linkcolor=darkgreen,
    filecolor=magenta, 
    citecolor=darkgreen}

\usepackage{exercise}
\usepackage{systeme}

\usepackage{tikz-cd}
\newcommand{\Addresses}{{
  \bigskip
  \footnotesize
 \ \ \ T. Brazelton, \textsc{Department of Mathematics, University of Pennsylvania}\par\nopagebreak
  \textit{E-mail address}: \texttt{tbraz@math.upenn.edu}

  \medskip
  
  J. Harrington, \textsc{Department of Mathematics, Cedar Crest College}\par\nopagebreak
  \textit{E-mail address}: \texttt{joshua.harrington@cedarcrest.edu}

  \medskip  
  M. Litman, \textsc{Department of Mathematics, University of California, Davis}\par\nopagebreak
  \textit{E-mail address}: \texttt{mclitman@ucdavis.edu}

  \medskip

  T.W.H. Wong, \textsc{Department of Mathematics, Kutztown University}\par\nopagebreak
  \textit{E-mail address}: \texttt{wong@kutztown.edu}
}}

\newcommand{\Z}{\mathbb{Z}}

\newcommand{\fR}{\mathfrak{R}}
\newcommand{\F}{\mathbb{F}}

\renewcommand{\epsilon}{\varepsilon}

\renewcommand{\emptyset}{\varnothing}

\newcommand{\floor}[1]{\left\lfloor #1 \right\rfloor}

\newcommand{\OB}[1]{\left\{ #1 \right\}}

\newcommand{\Hyp}{\mathcal{H}}
\newcommand{\Ell}{\mathcal{E}}
\newcommand{\Para}{\mathcal{P}}
\newcommand{\im}{\text{im}}

\newcommand{\legendre}[2]{\chi_{#2}({#1})}

\usepackage{cleveref}
\usepackage{thmtools}

\def\makeautorefname#1#2{\expandafter\def\csname#1autorefname\endcsname{#2}}
\makeautorefname{eqn}{Equation}%
\makeautorefname{sec}{Section}%
\makeautorefname{subsec}{Subsection}%
\makeautorefname{footnote}{footnote}%
\makeautorefname{item}{item}%
\makeautorefname{figure}{Figure}%
\makeautorefname{table}{Table}%
\makeautorefname{part}{Part}%
\makeautorefname{app}{Appendix}%
\makeautorefname{assump}{assumption}%
\makeautorefname{cla}{claim}%
\makeautorefname{conj}{conjecture}%
\makeautorefname{cor}{corollary}%
\makeautorefname{cex}{counterexample}%
\makeautorefname{cexs}{counterexamples}%
\makeautorefname{dig}{digression}%
\makeautorefname{disc}{discussion}%
\makeautorefname{def}{definition}%
\makeautorefname{ex}{example}%
\makeautorefname{exs}{examples}%
\makeautorefname{fac}{fact}%
\makeautorefname{intu}{intuition}%
\makeautorefname{lem}{lemma}%
\makeautorefname{nota}{notation}%
\makeautorefname{note}{note}%
\makeautorefname{prop}{proposition}%
\makeautorefname{rmk}{remark}%
\makeautorefname{term}{terminology}%
\makeautorefname{thm}{theorem}%
\makeautorefname{upsh}{upshot}%
\theoremstyle{definition}
\newtheorem{theorem}{Theorem}[section]
\newtheorem{assumption}[theorem]{Assumption}

\newtheorem{corollary}[theorem]{Corollary}

\newtheorem{definition}[theorem]{Definition}

\newtheorem{example}[theorem]{Example}

\newtheorem{lemma}[theorem]{Lemma}
\newtheorem{notation}[theorem]{Notation}

\newtheorem{proposition}[theorem]{Proposition}
\newtheorem{remark}[theorem]{Remark}

\makeatletter
\let\c@corollary=\c@theorem
\let\c@proposition=\c@theorem
\let\c@lemma=\c@theorem
\let\c@conjecture=\c@theorem
\let\c@definition=\c@theorem
\let\c@example=\c@theorem
\let\c@remark=\c@theorem
\let\c@notation=\c@theorem
\makeatother

\author[Brazelton]{Thomas Brazelton}
\author[Harrington]{Joshua Harrington}
\author[Litman]{Matthew Litman}
\author[Wong]{Tony W.H. Wong}

\title{Residue Sums of Dickson Polynomials over finite fields}
\date{\today}

\begin{document}
\maketitle

\begin{abstract}
    Given a polynomial with integral coefficients, one can inquire about the possible residues it can take in its image modulo a prime $p$. The sum over the distinct residues can sometimes be computed independent of the prime $p$; for example, Gauss showed that the sum over quadratic residues vanishes modulo a prime. In this paper we provide a closed form for the sum over distinct residues in the image of Dickson polynomials of arbitrary degree over finite fields of odd characteristic, and prove a complete characterization of the size of the image set. Our result provides the first non-trivial classification of such a sum for a family of polynomials of unbounded degree.
\end{abstract}

\tableofcontents

\section{Introduction}

For an integral polynomial $f$ and an odd prime $p$, we denote by $\fR_p(f)$ the image set of $f$ in $\F_p$, the finite field of order $p$.  Many properties of $\fR_p(f)$ have been well-studied if $f$ is of small degree. For example, it is well-known that $|\fR_p(f)|=(p+1)/2$ if $f$ is quadratic. Following the work of von Sterneck \cite{vonsterneck} and Kantor \cite{kantor} in the early 20th century, the size of the image set of a cubic polynomial was determined: if $p>3$ is prime, then
\begin{align*}
    \left| \fR_p(x^3 + ax^2 + bx + c) \right| &= \begin{cases}
    p & a^2 - 3b=0,\ p\equiv -1 \pmod{3}; \\
    \frac{p+2}{3} & a^2-3b = 0,\ p\equiv 1\pmod{3}; \\
    \frac{2p-1}{3} & a^2-3b\ne 0,\ p\equiv -1 \pmod{3}; \\
    \frac{2p+1}{3} & a^2-3b\ne 0,\ p\equiv 1\pmod{3}.
    \end{cases}
\end{align*}
For quartic and higher degree polynomials, not as much is known. Sun \cite{sun} investigated $|\fR_p(f)|$ for quartic polynomials $f$ with no cubic term, Chou, Gomez-Calderon, and Mullen \cite{CGM} established $|\fR_p(f)|$ for Dickson polynomials $f$ (we will discuss this in greater detail later), and Cusick \cite{cusick} investigated an infinite family of polynomials over a finite field of characteristic $2$. Uchiyama \cite{uchiyama} provided sufficient conditions for a polynomial $f$ to satisfy the lower bound $|\fR_p(f)|>p/2$, but noted that this does not hold in general. Just a few years later, Birch and Swinnerton-Dyer made Uchiyama's bound more precise \cite{BSD}, and this was further improved upon by Voloch \cite{voloch}. Probabilistic methods over finite fields allow one to ask about the ``average'' value of $|\fR_p(f)|$, varying over polynomials of degree $n$, and this has proven to be a fruitful direction of research (see for example \cite{cohen,knopfmacher}), however ascertaining the value $|\fR_p(f)|$ for arbitrary polynomials still appears intractable at the time of writing.

Another interesting property of $\fR_p(f)$ is the \emph{residue sum}, denoted by $S_p(f)$, defined to be the sum of the elements of $\fR_p(f)$ in $\F_p$. Gauss \cite{Gauss} first proved that $S_p(x^2)=0$. Considering $f(x)=x^2$ as a special case of polygonal numbers, it is natural to investigate the residue sum of triangular numbers modulo $p$, which was shown to be $-16^{-1}$ in $\F_p$ by Stetson \cite{Stetson} in 1904. In other words, Stetson showed that 
\[
    S_p\left(\frac{x(x+1)}{2}\right)=-\frac{1}{16}.
\]
This result was later generalized by Gross, Harrington, and Minott \cite{GHM}, who computed for $a\not\equiv 0\pmod{p}$ that
\[
    S_p\left(ax^2+bx+c\right)=-\frac{b^2-4ac}{8a}.
\]
We observe that the residue sums of quadratic polynomials are invariant for all odd primes $p$. Such is not true if $f$ has a higher degree. Finch-Smith, Harrington, and Wong \cite{FHW} showed that if $a\not\equiv 0\pmod{p}$ where $p>3$ is an odd prime, then 
$$S_p\left(ax^3+bx^2+cx+d\right)=\begin{cases} 
\dfrac{27a^2d-9abc+2b^3}{81a^2}\quad&\text{if }b^2\neq3ac\text{ and }p \equiv 1 \pmod{3};\vspace{3pt}\\
-\dfrac{27a^2d-9abc+2b^3}{81a^2}&\text{if }b^2\neq3ac\text{ and }p \equiv -1 \pmod{3};\vspace{3pt}\\
\dfrac{2\left(27a^2d-9abc+2b^3\right)}{81a^2}\quad&\text{if }b^2=3ac\text{ and }p \equiv 1 \pmod{3};\vspace{3pt}\\
0&\text{if }b^2=3ac\text{ and }p\equiv -1 \pmod{3}.
\end{cases}$$
While these results are interesting, the residue sums above have only been investigated for low-degree polynomials. In this article, we study $\fR_p(f)$ and $S_p(f)$ when $f$ is a \textit{Dickson polynomial}, which is an infinite family of polynomials with degrees that are arbitrarily large.

\begin{definition}
For a nonzero integer $a$, the Dickson polynomials $D_n(x,a)$ for $n\geq 0$ are defined recursively by $D_0(x,a)=2$, $D_1(x,a)=x$, and $D_n(x,a) = x D_{n-1}(x,a) - a D_{n-2}(x,a)$ for $n\ge 2$.
\end{definition}
The Dickson polynomials are ubiquitous in algebra and number theory. They are closely related to the \textit{Chebyshev polynomials} $T_n(x)$, and when $a=-1$, we recover the \textit{Lucas polynomials} $L_n(x)=D_n(x,-1)$. The Lucas polynomials are a ``polynomialization'' of the famous sequence of Lucas numbers, where the $n$th Lucas number can be obtained as $L_n(1)$.

As an illustrative example of residue sums, consider the Lucas polynomials at the prime $p=7$ as in Table~\ref{tab:table1}.
\begin{table}[h]\label{tab:table1}
\caption{Investigation of $S_7(L_n)$ for $1\le n\le 40$.}
\par\noindent\rule{0.7\textwidth}{0.4pt} \\
\vspace{0.2em}
\begin{tabular}{c | c}
$n$ & $S_7(L_n)$ \\
\midrule
 1 & 0 \\
 2 & 1 \\
 3 & 0 \\
 4 & 1 \\
 5 & 0 \\
 6 & 2 \\
 7 & 0 \\
 8 & 6
 \end{tabular}\quad
 \begin{tabular}{c | c}
$n$ & $S_7(L_n)$ \\
\midrule
 9 & 0 \\
 10 & 1 \\
 11 & 0 \\
 12 & 2 \\
 13 & 0 \\
 14 & 1 \\
 15 & 0 \\
 16 & 1
  \end{tabular}\quad
 \begin{tabular}{c | c}
$n$ & $S_7(L_n)$ \\
\midrule
 17 & 0 \\
 18 & 2 \\
 19 & 0 \\
 20 & 1 \\
 21 & 0 \\
 22 & 1 \\
 23 & 0 \\
 24 & 0
 \end{tabular}\quad
 \begin{tabular}{c | c}
$n$ & $S_7(L_n)$ \\
\midrule
 25 & 0 \\
 26 & 1 \\
 27 & 0 \\
 28 & 1 \\
 29 & 0 \\
 30 & 2 \\
 31 & 0 \\
 32 & 1
\end{tabular}\quad
\begin{tabular}{c | c}
$n$ & $S_7(L_n)$ \\
\midrule
 33 & 0 \\
 34 & 1 \\
 35 & 0 \\
 36 & 2 \\
 37 & 0 \\
 38 & 1 \\
 39 & 0 \\
 40 & 6
\end{tabular}
\centering
\end{table}
We remark that the residue sum $S_p(L_n)$ has a very limited number of possible values. Shockingly, this is not a property that is special to the case $p=7$.  As a consequence of our study on the Dickson polynomials, we can provide a complete classification of $S_p(L_n)$ for all odd primes $p$ and positive integers $n$ which shows that $S_p(L_n)\in\{-1,0,1,2\}$. In fact, the following theorem implies that $S_p(D_n(x,a))\in \{-2a^{n/2},-a^{n/2},0,a^{n/2},2a^{n/2}\}$. As all our results hold for finite fields of odd characteristic, we state them in that generality where $q=p^k$ for $p$ an odd prime. 

We denote by $\legendre{\cdot}{q}$ the quadratic character over $\F_q$, that is the multiplicative function defined by
\begin{equation}\label{eqn:quadratic-character}
\begin{aligned}
    \legendre{a}{q} := a^{\frac{q-1}{2}}= \begin{cases}
        0 & \mathrm{if\ }a=0\in\F_q; \\
        1 & \mathrm{if\ }a\ \mathrm{is\ a\ quadratic\ residue\ in\ } \F_q; \\
        -1 & \mathrm{if\ }a\ \mathrm{is\ not\ a\ quadratic\ residue\ in\ } \F_q.
        \end{cases}
\end{aligned}
\end{equation}
This function generalizes the Legendre symbol over a field of prime order. Furthermore, it is natural to extend the definitions of $\fR_p$ and $S_p$ to the field of order $q$ and denote these by $\fR_q$ and $S_q$.

\begin{theorem}\label{thm:S-sum} Let $a$ be an integer, $n$ be a nonnegative integer, and $q$ be an odd prime power such that $a\neq 0 \in \F_q$. Let $d=\gcd(n,q-1)$ and $\delta = \gcd(n,q+1)$, and let $r$ be the highest power of 2 dividing $q^2-1$. Then the sum of the elements in the image of the Dickson polynomials is\footnote{When $(q-1)\mid n$ we remark that $\legendre{a}{q}^{n/d} a^{n/2} = a^n$, while when $(q+1)\mid n$ we have that $\legendre{a}{q}^{n/\delta}a^{n/2} = a^{n-n/\delta}$. We keep it in the stated form to highlight the symmetry.}
\begin{align*}
    S_q(D_n(x,a)) &= \begin{cases}
    0 & 2^{r-1}\mid n; \\
    -\legendre{a}{q}^{\frac{n}{d}+\frac{n}{\delta}} a^{n/2} & \text{else},
    \end{cases} 
    + \begin{cases}
    \legendre{a}{q}^{\frac{n}{d}}a^{n/2} & (q-1)\mid n; \\ 0 & \text{else},
    \end{cases} 
    + \begin{cases} \legendre{a}{q}^{\frac{n}{\delta}}a^{n/2} & (q+1)\mid n; \\
    0 & \text{else}.
    \end{cases}
\end{align*}
\end{theorem}

\begin{corollary} In the situation above, we have five possibilities for the residue sum, as $n$ and $a$ vary over all integers, and $q$ over all odd prime powers
\[
S_q(D_n(x,a)) \in \left\{ 0, \pm a^{n/2}, \pm 2a^{n/2} \right\}.
\]
\end{corollary}

This demonstrates the first non-trivial classification for $S_q(f)$ where $f$ varies over an infinite family of polynomials of unbounded degree. In the process of proving this theorem, we provide a complete characterization of the size of $\mathfrak{R}_q(D_n(x,a))$ for all $n$, $a$, and odd prime powers $q$, which is the main result of Chou, Gomez-Calderon, and Mullen \cite[Theorem 10]{CGM} for odd characteristic.

\begin{theorem}
Let $p$ be an odd prime power, $n$ an even natural number, $d=\gcd(q-1,n)$, $\delta=\gcd(2(q+1),n)$, and $2^r$ the highest power of 2. Then the size of the value set of the $n$th Dickson polynomial over $\F_q$ is
\begin{align*}
    |\mathfrak{R}_q(D_n(x,a))|= \floor{\frac{q-1}{2d}}+\floor{\frac{q+1}{2\delta}}+1+\begin{cases}
    1
& \mathrm{if\ }\legendre{a}{q}=-1 \mathrm{\ and\ } 2^{r-1}\mid\mid n; \\
0 & \mathrm{otherwise}.\end{cases}
\end{align*}
\end{theorem}

\subsection*{Acknowledgements}
The first named author is supported by an NSF Graduate Research Fellowship (DGE-1845298). 

\section{Preliminaries}

In this section, we present some preliminary results and notation that will be useful in our investigation.

\begin{notation} Throughout this paper we will fix $p$ to be an odd prime, and $q$ to be some power of it, defining a finite field $\F_q$. We fix a primitive $(q^2-1)$st root of unity $\zeta_{q^2-1}$ to be a generator of the group of units $\F_{q^2}^\times$. For each positive factor $m$ of $q^2-1$, let $\zeta_m=\zeta_{q^2-1}^{(q^2-1)/m}$. In particular, $\zeta_{q-1}\in\F_q^\times$ is a primitive $(q-1)$st root of unity. When we consider a Dickson polynomial over $\F_q$ and $a$ nonzero in $\F_q$, let $A$ be the smallest positive integer such that $a=\zeta_{q-1}^A$. Finally, we denote by $S_q(f)$ the residue sum of an integral polynomial $f(x)$ over the finite field $\F_q$.
\end{notation}

\subsection{Dickson polynomials}

Using standard methods of solving recurrence relations, one can show that the Dickson polynomials admit a Binet formula expansion:
    \begin{align}\label{eq:Ln}
        D_n(x,a) = \omega(x,a)^n + \overline{\omega}(x,a)^n,
    \end{align}
where    
    \begin{align*}
        \omega(x,a) = \frac{x+\sqrt{x^2-4a}}{2}\quad\text{and}\quad\overline{\omega}(x,a) = \frac{x-\sqrt{x^2-4a}}{2}.
    \end{align*}
Using the expressions for $\omega$ and $\overline{\omega}$, we make note of the following properties: 
\begin{equation}\label{eqn:omega-properties}
\begin{aligned}
    x &= \omega(x,a) + \overline{\omega}(x,a), \\
    a &= \omega(x,a)\overline{\omega}(x,a).
\end{aligned}
\end{equation}
Since $a \neq 0$, from this expression we see $\overline{\omega}(x,a) = a\omega(x,a)^{-1}$.

\begin{example}
One may check that the small index Dickson polynomials are given by
\allowdisplaybreaks
\begin{align*}
    D_0(x,a) &= 2           &D_4(x,a) &=x^4+4x^2a + 2a^2   \\
    D_1(x,a) &= x               &D_5(x,a) &= x^5 + 5x^3a + 5xa^2\\
    D_2(x,a) &= x^2 - 2a                &D_6(x,a) &= x^6 - 6x^4a +9x^2a^2 -2a^3\\
    D_3(x,a) &= x^3 - 3xa               &D_7(x,a) &=x^7-7x^5a+14x^3a^3-7xa^3.
\end{align*}
\end{example}

\begin{proposition}\label{prop:n-odd-case}
If $n$ is odd, then $S_q(D_n(x,a))=0$.
\end{proposition}
\begin{proof}
It follows from the recursive definition of $D_n(x,a)$ that if $n$ is odd, then $D_n(x,a)$ is an odd polynomial.  Consequently, if $y\in\fR_q(D_n(x,a))$, then $-y\in\fR_q(D_n(x,a))$.  Since $p$ is odd, $y\not\equiv -y$ in $\F_q$, and we deduce that $S_q(D_n(x,a))=0$.
\end{proof}

\begin{proposition}\label{prop:a0modp} Suppose that $a\equiv 0$ in $\F_q$. Then we have that
\begin{align*}
    S_q(D_n(x,a)) &= \begin{cases} 1 & n=q-1; \\
    0 & \text{else}. \end{cases}
\end{align*}
\end{proposition}
\begin{proof} Via the recursive relation of the Dickson polynomials, we have that $D_n(x,a) \equiv x^n$ whenever $a$ vanishes over $\F_q$. From this the problem reduces to summing $n$th powers over a finite field.
\end{proof}

\begin{assumption}\label{ass:ass} As \autoref{prop:n-odd-case} determines the residue sum $S_q(D_n(x,a))$ for all odd $n$, for the remainder of this paper, we will make the standing assumption that $n$ is even. Additionally \autoref{prop:a0modp} determines the case where $a\equiv 0 \in \F_q$, so we will assume without loss of generality that $a\not\equiv 0 \in \F_q$. Finally we will make the standing assumption that $q\ne 3$, however one may check by direct computation that \autoref{thm:S-sum} holds when $q=3$.
\end{assumption}
Using $\overline{\omega}(x,a)=a\omega(x,a)^{-1}$ , we simplify \autoref{eq:Ln} to 
    \begin{align*}
        D_n(x,a) = \omega(x,a)^n + a^n\omega(x,a)^{-n},
    \end{align*}
and we exploit this form of $D_n$ to prove the following proposition and other results throughout the paper.

\begin{proposition}\label{prop:dickson-poly-is-2-iff-omegan-is-1} Let $x,y\in\F_q$ be arbitrary.  Then $D_n(x,a) = D_n(y,a)$ if and only if $\omega(x,a)^n = \omega(y,a)^{n}$ or $\omega(x,a)^n=\overline{\omega}(y,a)^n=a^n\omega(y,a)^{-n}$.
\end{proposition}
\begin{proof} Suppose that $D_n(x,a) = \omega(x,a)^n + a^n\omega(x,a)^{-n} = \omega(y,a)^n + a^n\omega(y,a)^{-n} = D_n(y,a)$. By multiplying both sides of the equation by $\omega(x,a)^n$ and rearranging the terms, we have that
\begin{align*}
    \omega(x,a)^{2n} - \left(\omega(y,a)^n + a^n\omega(y,a)^{-n}\right)\omega(x,a)^n + a^n &= 0.
\end{align*}
That is, $\omega(x,a)^n$ is a solution of the polynomial $$t^2 - \left(\omega(y,a)^n + a^n\omega(y,a)^{-n}\right)t + a^n = \left(t-\omega(y,a)^n\right)\left(t-a^n\omega(y,a)^{-n}\right).$$
\end{proof}

\begin{proposition} \cite[Lemma~7]{CGM} Let $x\in \mathbb{F}_q$. Then we have that $\omega(x,a)^n = \bar{\omega}(x,a)^n$ if and only if $D_n(x,a) = \pm 2a^{n/2}$.
\end{proposition}

\begin{corollary}
For any $x\in \F_q$, we have that $D_n(x,a) = \pm2 a^{n/2}$ if and only if $\omega(x,a)^n = \pm a^{n/2}$.
\end{corollary}

\subsection{Hyperbolic, elliptic, and parabolic elements}

\begin{notation}
We partition $\F_q$ into three subsets, denoted by
\begin{align*}
\mathcal{H}_q(a) & = \left\{ x\in\mathbb{F}_q \colon \chi_q(x^2-4a)=1 \right\},\\
\mathcal{E}_q(a) &= \left\{x\in\mathbb{F}_p \colon \chi_q(x^2-4a)= -1 \right\},\\
\mathcal{P}_q(a) &= \left\{x\in\mathbb{F}_p \colon \chi_q(x^2-4a)=0 \right\}.
\end{align*}
We will refer to elements of $\mathcal{H}_q(a)$, $\mathcal{E}_q(a)$, and $\mathcal{P}_q(a)$ as \emph{hyperbolic}, \emph{elliptic}, and \emph{parabolic}, respectively. This terminology is inspired by the work of Bourgain, Gamburd, and Sarnak \cite{bgs} on showing the connectivity of the Markoff mod $p$ graphs.
\end{notation}

Our understanding of $\mathfrak{R}_q(D_n(x,a))$ will come from investigating the images of these three sets under the map $D_n$. We will denote by $D_n(\mathcal{H}_q(a),a)\subseteq\F_q$ the image set of the hyperbolic elements under the Dickson polynomial, and similarly for the elliptic and parabolic sets. In order to compute the residue sum $S_q(D_n(x,a))$, it will suffice to have a handle on these three image sets as well as their potential overlaps.

\begin{remark}\label{rmk:preliminary-observations}
We note here some preliminary observations about the quantity $\omega(x,a)$.
\begin{enumerate}
    \item\label{rmk:preliminary-observations(1)} If $x\in \mathcal{H}_q(a)$, then $\omega(x,a)$ is an element of $\F_q$. In particular, $\omega(x,a)^{q-1} =1$, thus we may write $\omega(x,a) = \zeta_{q-1}^c$ for some $c$, where $\zeta_{q-1}$ is our fixed primitive $(q-1)$st root of unity.

    \item If $x\in \mathcal{E}_q(a)$, then $\omega(x,a)$ is an element of $\F_{q^2}$ but not an element of $\F_q$. Thus we have that $\omega(x,a)^{q^2-1} = 1$.

    \item Observe that $\mathcal{P}_q(a)$ is nonempty if and only if $\legendre{a}{q}=1$, where we recall that $\chi_q$ is the quadratic character as in \autoref{eqn:quadratic-character}. In this situation, if $x\in \mathcal{P}_q(a)$, then $\omega(x,a) = x/2$ is an element of $\F_q$. Moreover, we have that $x=\pm\sqrt{4a}=\pm2\sqrt{a}$. Now, since $n$ is even,
\begin{align*}
    D_n(x,a) &= \omega(x,a)^n + \overline{\omega(x,a)}^{n} = \left( \frac{\pm 2\sqrt{a}}{2} \right)^n + \left( \frac{\pm 2\sqrt{a}}{2} \right)^{n} = 2 a^{n/2}.
\end{align*}
\end{enumerate}
\end{remark}
We now establish the following property of elliptic elements.
\begin{proposition}\label{prop:omega-properties-for-elliptic-elts} For all elliptic elements $x\in\mathcal{E}_q(a)$, we have that $\omega(x,a)^{q+1}= a \in \F_q$. In particular, we have that $\omega(x,a)=\zeta_{q^2-1}^{A+k(q-1)}$ in $\F_{q^2}$ for some integer $k$.
\end{proposition}
\begin{proof}
Via the freshman's dream, we may write $\omega(x,a)^q$ as
\begin{align*}
    \omega(x,a)^q &= \frac{x + \left( \sqrt{x^2- 4a} \right)^q}{2},
\end{align*}
and we observe that
\begin{align*}
     \overline{\omega}(x,a) &= \frac{x - \sqrt{x^2-4a}}{2}.
\end{align*}
As $x$ is not parabolic, the quantity $x^2 - 4a$ is nonvanishing, thus we have that $\left( x^2 - 4a \right)^{q-1} = 1$ in $\F_q$. As $\sqrt{x^2-4a}$ is not defined over $\F_q$, it is not fixed by the Frobenius endomorphism on $\F_{q^2}$.
This implies that $\left( x^2-4a \right)^{\frac{q-1}{2}} = -1$ in $\F_{q^2}$. Thus we see that
\begin{align*}
    \omega(x,a)^q &= \frac{x + \left( \sqrt{x^2-4a} \right)^q}{2} = \frac{x + \left( x^2-4a \right)^{\frac{q-1}{2}} \sqrt{x^2 - 4a}}{2} = \frac{x - \sqrt{x^2-4a}}{2} = \overline{\omega}(x,a).
\end{align*}
Therefore $\omega(x,a)^{q+1} = \omega(x,a)\overline{\omega}(x,a)=a$.

By \autoref{rmk:preliminary-observations}, we have that $\omega(x,a) = \zeta_{q^2-1}^c$ for some $c$. From the observation that $\omega(x,a)^{q+1} = a = \zeta_{q-1}^A = \zeta_{q^2-1}^{(q+1)A}$, we must have that $c(q+1) \equiv A(q+1) \pmod{q^2-1}$. That is, $c = A + k(q-1)$ for some integer $k$, which we may assume to lie in the range $1\le k \le q+1$, since we only care about the residue of $c$ modulo $q^2-1$.
\end{proof}

We can now state explicitly what each set in the partition of $\F_q$ looks like:

\begin{proposition}\label{prop:hyp-ell-para-sets-description}
The hyperbolic, elliptic, and parabolic sets over the finite field $\mathbb{F}_q$ are given by:
\allowdisplaybreaks
\begin{align*}
        \Hyp_q(a) &= \OB{\zeta_{q-1}^c + \zeta_{q-1}^{A-c} \colon 1\le c\le q-1\ \mathrm{and }\ 2c\not\equiv A \pmod{q-1} }, \\
        \Ell_q(a) &= \OB{\zeta_{q^2-1}^{A+k(q-1)}+\zeta_{q^2-1}^{Aq-k(q-1)} \colon 1\le k\le q+1\ \mathrm{and}\ 2k \not\equiv A \pmod{q+1}   }, \\
        \Para_q(a) &= \begin{cases}
        \OB{\pm 2 a^{1/2}} & \mathrm{if }\ \legendre{a}{q}=1; \\
        \emptyset & \mathrm{if }\ \legendre{a}{q}=-1.
        \end{cases}
    \end{align*}
\end{proposition}
\begin{proof}
For $x$ hyperbolic, we know that $\omega(x,a) = \zeta_{q-1}^c$ for some $c$ by \autoref{rmk:preliminary-observations}. We should see for which $c$ we are getting hyperbolic elements. Since
\begin{align*}
    \sqrt{x^2-4a} = \omega(x,a) - \bar{\omega}(x,a),
\end{align*}
we can check whether this quantity is defined over $\mathbb{F}_q$ (meaning that $x^2-4a$ is a residue). This is equivalent to checking that it is fixed under the Frobenius. Note that
\begin{align*}
    \omega(x,a) - \bar{\omega}(x,a) = \zeta_{q-1}^c - \zeta_{q-1}^{A-c}.
\end{align*}
Applying the Frobenius, we see
\begin{align*}
    \left( \zeta_{q-1}^c - \zeta_{q-1}^{A-c} \right)^q &= \zeta_{q-1}^{qc} - \zeta_{q-1}^{q(A-c)} = \zeta_{q-1}^{c} - \zeta_{q-1}^{A-c}.
\end{align*}
Thus for any $c$, we have that $\zeta_{q-1}^c + \zeta_{q-1}^{A-c}$ gives an element for which $\sqrt{x^2-4a} \in \mathbb{F}_q$. We should verify that it is not accidentally producing a parabolic element, i.e. that we are not accidentally getting $\sqrt{x^2-4a} = 0$. This would occur for some $c$ if
\begin{align*}
    \sqrt{x^2-4a} &= \zeta_{q-1}^c - \zeta_{q-1}^{A-c} = 0,
\end{align*}
that is, if $2c\equiv A \pmod{q-1}$.

For elliptic elements, we want to verify that $\sqrt{x^2-4a}$ is not defined over $\mathbb{F}_q$, equivalently that it is not fixed under the Frobenius. So we want to throw out any $k$ for which
\begin{align*}
    \left( \zeta_{q^2-1}^{A + k(q-1)} - \zeta_{q^2-1}^{Aq-k(q-1)} \right) &= \left( \zeta_{q^2-1}^{A + k(q-1)} - \zeta_{q^2-1}^{Aq-k(q-1)} \right)^q.
\end{align*}
This would give us the equality
\begin{align*}
    \zeta_{q^2-1}^{A + k(q-1)} - \zeta_{q^2-1}^{Aq-k(q-1)} &= \zeta_{q^2-1}^{Aq + k(q^2-q)} - \zeta_{q^2-1}^{Aq^2-k(q^2-q)} \\
    &= \zeta_{q^2-1}^{Aq + k(1-q)} - \zeta_{q^2-1}^{A - k(1-q)}.
\end{align*}
Rearranging, we see that this is the same as
\begin{align*}
    2\zeta_{q^2-1}^{A+k(q-1)} &= 2\zeta_{q^2-1}^{Aq - k(q-1)}.
\end{align*}
Since $2$ is invertible in $\mathbb{F}_{q^2}$ we are left with the congruence
\begin{align*}
    A + k(q-1) \equiv Aq - k(q-1) \pmod{q^2-1}.
\end{align*}
This is equivalent to $2k \equiv A \pmod{q+1}$. So we must omit these $k$'s out in order to ensure we are getting an elliptic element.
\end{proof}

We will be interested in the images of the hyperbolic, elliptic, and parabolic sets under the Dickson polynomial $D_n(x,a)$. In particular if we can understand the images over these sets, as well as their potential intersection, then we can completely understand $\im(D_n(x,a))$.

\begin{lemma}\label{lem:CGM} \cite[Lemma~8]{CGM} Let $x,y\in \mathbb{F}_q^\times$, and let $x = u + a/u$ and $y = v + a/v$, where $u \in \mathbb{F}_q^\times$, and $v \in \mathbb{F}_{q^2}^\times$ so that $v^{q+1} = a$. Then if $u^n = v^n$ for some $n\ge 0$, this implies that
\begin{align*}
    u^n = a^n/u^n = v^n = a^n/v^n.
\end{align*}
In particular they are all equal to $a^{n/2}$ or $-a^{n/2}$.
\end{lemma}

This result allows us to restrict the values of any possible overlap in the hyperbolic and elliptic images.
\begin{proposition}\label{prop:possible-hyp-ellip-overlap} There are only two possible values for the intersection of the hyperbolic and elliptic images, namely
\begin{align*}
    D_n \left( \mathcal{H}_q(a),a \right) \cap D_n \left( \mathcal{E}_q(a) ,a\right) \subseteq \left\{ \pm 2a^{n/2} \right\}.
\end{align*}
\end{proposition}

\begin{proof}
Suppose that $x\in \mathcal{H}_q(a)$ and $y\in \mathcal{E}_q(a)$ so that $D_n(x,a) = D_n(y,a)$. Then there are some $c$ and $k$ for which
\begin{align*}
    \omega(x,a)^n &= \zeta_{q-1}^{nc} \\
    \omega(y,a)^n &= \zeta_{q^2-1}^{n(A+k(q-1)}.
\end{align*}
In order to have $D_n(x,a) = D_n(y,a)$, by \autoref{prop:dickson-poly-is-2-iff-omegan-is-1} we have that $\omega(x,a)^n = \omega(y,a)^n$ or $\omega(x,a)^n = \bar{\omega}(y,a)^n$.

In the first case, suppose that $\omega(x,a)^n = \omega(y,a)^n$. Since $y$ is elliptic, we have that $\omega(y,a)^{q+1} = a$ by \autoref{prop:omega-properties-for-elliptic-elts}. Therefore by invoking \autoref{lem:CGM} using $u = \omega(x,a)$ and $v = \omega(y,a)$, we have that $\omega(x,a)^n = \omega(y,a)^n =\pm a^{n/2}$. In particular this implies that
\begin{align*}
    D_n(x,a) = D_n(y,a) = \pm 2a^{n/2}.
\end{align*}

In the latter case, if $\omega(x,a)^n = \bar{\omega}(x,a)^n$, we can observe that
\begin{align*}
    \bar{\omega}(y,a)^{q+1} &= \frac{a^{q+1}}{\omega(y,a)^{q+1}} = \frac{a^{q+1}}{a} = a^q = a.
\end{align*}
Invoking \autoref{lem:CGM} with $v = \bar{\omega}(y,a)$, we have then that
\begin{align*}
    \omega(x,a)^n = \bar{\omega}(y,a)^n = \pm a^{n/2},
\end{align*}
and therefore $D_n(x,a) = D_n(y,a) = \pm 2a^{n/2}$.
\end{proof}

\section{Evaluation of the residue sum}
As remarked earlier, our strategy for studying the residue sum $S_q(D_n(x,a))$ will be to investigate the sum over the hyperbolic, elliptic, and parabolic sets, as well as over their overlaps. To this end, we introduce some new notation: for any subset $B \subseteq \F_q$, we denote by $S_q^B(D_n(x,a))$ the sum over the distinct elements in $D_n(B,a)$. If $C \subseteq \F_q$ is another subset, we denote by $S_q^{B, C}(D_n(x,a))$ the sum over distinct elements of $D_n(B,a)\cap D_n(C,a)$, and we have similar notation for triple intersections. In this notation, the total sum will be computed as
\allowdisplaybreaks
\begin{align*}
    S_q(D_n(x,a)) &= S_q^{\mathcal{H}_q(a)}(D_n(x,a)) + S_q^{\mathcal{E}_q(a)}(D_n(x,a)) + S_q^{\mathcal{P}_q(a)}(D_n(x,a)) \\
    & \quad - S_q^{\mathcal{H}_q(a), \mathcal{E}_q(a)}(D_n(x,a)) - S_q^{\mathcal{H}_q(a), \mathcal{P}_q(a)}(D_n(x,a)) \\
& \quad - S_q^{\mathcal{E}_q(a), \mathcal{P}_q(a)}(D_n(x,a)) + S_q^{\mathcal{H}_q(a), \mathcal{E}_q(a), \mathcal{P}_q(a)}(D_n(x,a)).
\end{align*}

Our preliminary observations about the quantities $\omega(x,a)$ as $x$ varies over the hyperbolic and elliptic sets indicate that elements in $D_n(\mathcal{H}_q(a),a)$ and $D_n(\mathcal{E}_q(a),a)$ will be able to be characterized using roots of unity defined over $\F_q$ or its quadratic extension $\F_{q^2}$.

\begin{notation}\label{nota:d-delta} We will see that the residue sums $S_q(D_n(x,a))$ in \autoref{thm:S-sum} depend upon various properties of $n$, in particular the highest power of 2 dividing $n$ and the order of $n$ in $\F_q^\times$ and $\F_{q^2}^\times$ (which relates to divisors shared between $n$ and $q-1$ and $q+1$). To that end, we fix some notation:
\begin{align*}
    d &:= \gcd(n,q-1)  \qquad  m := \frac{n}{d}\\
    \delta &:= \gcd\left(n,q+1\right) \qquad \mu := \frac{n}{\delta}.
\end{align*}
We will also let $2^h$ denote the highest power of 2 dividing $q-1$, $2^\ell$ denote the highest power of $2$ dividing $q+1$, and $2^r$ the highest power of 2 dividing $q^2-1$. 
\end{notation}

We remark the following relationship between $2^r$ and the divisors $d$ and $\delta$ which will come in handy throughout our computations.
\begin{proposition}\label{prop:r-d-delta-relationship} Let $d$, $\delta$, $h$, $\ell$, and $r$ be as in \autoref{nota:d-delta}
\begin{enumerate}
    \item We have $\frac{q-1}{d}$ is odd if and only if $2^h\mid n$.
    \item We have $\frac{q+1}{\delta}$ is odd if and only if $2^\ell \mid n$.
    \item Both $\frac{q-1}{d}$ and $\frac{q+1}{\delta}$ are odd if and only if $2^{r-1} \mid n$.
    \item Both $\frac{n}{d}$ and $\frac{n}{\delta}$ are even if and only if $2^r \mid n$.
\end{enumerate}
\end{proposition}
\begin{proof} (1) and (2) follow directly from the definition of $h$ and $\ell$. 

As for (3), we notice that one of $\frac{q+1}{2}$ or $\frac{q-1}{2}$ will be odd, and therefore $h$ and $\ell$ cannot both be strictly greater than one. In particular, this tells us that $\max \left\{ h,\ell \right\} = h+\ell-1 = r-1$, from which the result follows.

For the forward direction of (4), let $2^s \mid\mid n$. Then $\frac{n}{d}$ even implies that $s>h$ and $\frac{n}{\delta}$ even implies that $s>\ell$. In particular $s>\max\{h,\ell\} = r-1$, and hence $s\ge r$.

For the backwards direction of (4), if $2^r\mid n$, then since $r = h+\ell$ and $h,\ell\ge 1$, we have that $2^{h+1}\mid n$ and $2^{\ell+1}\mid n$, implying that both $\frac{n}{d}$ and $\frac{n}{\delta}$ are even.
\end{proof}

It will also benefit us to record some parity constraints that can occur on these values. We will refer back to this result frequently.

\begin{proposition}\label{prop:parity-constraints} Continuing our notation from above:
\begin{enumerate}
    \item Both $\frac{q-1}{d}$ and $\frac{q+1}{\delta}$ cannot be even.
    \item If both $\frac{q-1}{d}$ and $\frac{q+1}{\delta}$ are odd, then we cannot have both $\frac{n}{d}$ and $\frac{n}{\delta}$ odd.
    \item If $2^{r-1}\mid\mid n$, then $\frac{n}{d}$ and $\frac{n}{\delta}$ have opposite parity.
\end{enumerate}
\end{proposition}
\begin{proof} The first result follows from the fact that $q$ is an odd prime power, hence one of $\frac{q-1}{2}$ or $\frac{q+1}{2}$ must be odd. In particular since $2|d$ and $2\mid \delta$, one of $\frac{q-1}{d}$ and $\frac{q+1}{\delta}$ must be odd.

For the second observation, we remark that $4\mid (q-1)$ or $4\mid (q+1)$. This implies that either $4\mid d$ or $4\mid \delta$ (since we are assuming both $\frac{q-1}{d}$ and $\frac{q+1}{\delta}$ are odd), and therefore $4\mid n$. However, we must have that $2\mid\mid (q-1)$ or $2\mid\mid (q+1)$, and therefore $2\mid\mid d$ or $2\mid\mid \delta$. In particular there are more powers of $2$ dividing $n$ than divide one of $d$ or $\delta$, and therefore at least one of $\frac{n}{d}$ or $\frac{n}{\delta}$ must be even.

For the third observation, we have by \autoref{prop:r-d-delta-relationship} that $2^{r-1}\mid n$ is equivalent to both $\frac{q-1}{d}$ and $\frac{q+1}{\delta}$ being odd. However $2^r\nmid n$ means that $\frac{n}{d}$ and $\frac{n}{\delta}$ cannot both be even. Via observation (2) of this proposition, they cannot both be odd, therefore they must have opposite parity.
\end{proof}

\subsection{Summing over the hyperbolic and elliptic images}\label{subsec:hyp-sum} 
Using the characterization of the hyperbolic and elliptic sets in \autoref{prop:hyp-ell-para-sets-description}, we can understand the hyperbolic and elliptic images, and therefore their sums.

We first treat the hyperbolic case. Via the Binet formula expansion, we may see that
\begin{equation}\label{eqn:hyperbolic-image}
\begin{aligned}
    D_n(\mathcal{H}_p(a),a) &= \left\{ \zeta_{\frac{q-1}{d}}^{m c} + \zeta_{\frac{q-1}{d}}^{m(A-c)} \colon 1\le c \le q-1,\ 2c\not\equiv A \pmod{q-1} \right\}.
\end{aligned}
\end{equation}
We remark that the residue of $c$ modulo $\frac{q-1}{d}$ matters when recording elements in the hyperbolic image, however the condition $2c\not\equiv A \pmod{q-1}$ is not equivalent to the condition $2c\not\equiv A \pmod{\frac{q-1}{d}}$. So we can have elements $c$ so that $2c\equiv A \pmod{\frac{q-1}{d}}$, but $2c\not \equiv A \pmod{q-1}$. This is how elements like $\pm 2a^{n/2}$ can appear in the hyperbolic image. In order to deal with this, we can provide an alternative description of the hyperbolic image.

\begin{proposition} The hyperbolic image can be described as
\begin{align*}
    &\left\{ \zeta_{\frac{q-1}{d}}^{m c} + \zeta_{\frac{q-1}{d}}^{m (A-c)} \colon 2c\not\equiv A \bmod{\frac{q-1}{d}} \right\}_{c=1}^{\frac{q-1}{d}} \cup \left\{ 2\zeta_{q-1}^{n c} \colon \substack{ 2c\equiv A \bmod{\frac{q-1}{d}} \text{ but} \\ 2c\not \equiv A \bmod{q-1}  }\right\}_{c=1}^{\frac{q-1}{d}}
\end{align*}
\end{proposition}

Thus to characterize the hyperbolic image, it suffices to understand when these congruences can be solved. As we see in the following proposition, this depends upon the parity of $A$ and $\frac{q-1}{d}$, as well as whether or not $d=2$.
\begin{proposition}\label{prop:hyperbolic-image} We have that the hyperbolic image $D_n(\mathcal{H}_q(a),a)$ is equal to
    \begin{align*}
     \begin{cases}
    \left\{2 \legendre{a}{q}^{n/d} a^{n/2}\right\} & d=q-1; \\
    \left\{ \zeta_{\frac{q-1}{d}}^{m c} + \zeta_{\frac{q-1}{d}}^{m (A-c)} \right\}_{c=1}^{\frac{q-1}{d}} & A\text{ odd, } \frac{q-1}{d} \text{ even}; \\
    \left\{ \zeta_{\frac{q-1}{d}}^{m c} + \zeta_{\frac{q-1}{d}}^{m (A-c)} \colon c\not\equiv \frac{1}{2}\left( A + \frac{q-1}{d} \right) \bmod{\frac{q-1}{d}} \right\}_{c=1}^{\frac{q-1}{d}} \cup \left\{  2(-1)^{n/d} a^{n/2} \right\} & A\text{ odd, } \frac{q-1}{d} \text{ odd}; \\
    \left\{ \zeta_{\frac{q-1}{d}}^{m c} + \zeta_{\frac{q-1}{d}}^{m (A-c)} \colon c\not\equiv \frac{A}{2} \bmod{\frac{q-1}{d}} \right\}_{c=1}^{\frac{q-1}{d}} & A\text{ even, } \frac{q-1}{d} \text{ odd, } d=2; \\
   \left\{ \zeta_{\frac{q-1}{d}}^{m c} + \zeta_{\frac{q-1}{d}}^{m (A-c)} \colon c\not\equiv \frac{A}{2} \bmod{\frac{q-1}{d}} \right\}_{c=1}^{\frac{q-1}{d}} \cup \left\{ 2 a^{n/2} \right\} & A\text{ even, } \frac{q-1}{d} \text{ odd, } d\ne 2; \\
    \left\{ \zeta_{\frac{q-1}{d}}^{m c} + \zeta_{\frac{q-1}{d}}^{m (A-c)} \colon c\not\equiv \frac{A}{2}, \frac{A}{2} + \frac{q-1}{2d} \bmod{\frac{q-1}{d}} \right\}_{c=1}^{\frac{q-1}{d}} \cup \left\{ -2a^{n/2} \right\} & A\text{ even, } \frac{q-1}{d} \text{ even, } d=2;    \\
   \left\{ \zeta_{\frac{q-1}{d}}^{m c} + \zeta_{\frac{q-1}{d}}^{m (A-c)} \colon c\not\equiv \frac{A}{2}, \frac{A}{2} + \frac{q-1}{2d} \bmod{\frac{q-1}{d}} \right\}_{c=1}^{\frac{q-1}{d}} \cup \left\{ 2a^{n/2},\ -2a^{n/2}\right\} & A\text{ even, } \frac{q-1}{d} \text{ even, } d\ne 2.
   \end{cases}
\end{align*}
\end{proposition}
\begin{proof} We know that solutions to $2c\equiv A \pmod{\frac{q-1}{d}}$ exist if and only if $\gcd \left( 2, \frac{q-1}{d} \right)$ divides $A$, and in this setting there are precisely $\gcd \left( 2, \frac{q-1}{d} \right)$ such solutions.
\begin{enumerate}
    \item In the case that $d=q-1$, we have that the hyperbolic image is simply $\{2\}$. However, since $d=q-1$, we can write $n = \frac{n}{d}(q-1)$, from which we can see that 
    \begin{align*}
    a^{n/2} &= \left( a^{\frac{q-1}{2}} \right)^{\frac{n}{d}} = \legendre{a}{q}^{n/d}.
\end{align*}
Since these are both congruent to $\pm 1$, they square to 1, so we may rewrite $2 = 2 \legendre{a}{q}^{n/d} a^{n/2}$. When discussing potential overlap in the hyperbolic and elliptic images later, it will benefit us to characterize the hyperbolic image in this a priori more convoluted form.
    
    \item In this case $\gcd \left( 2, \frac{q-1}{d} \right)$ is even, which does not divide $A$ since it is odd. Thus there are no solutions.

    \item If $A$ is odd and $\frac{q-1}{d}$ is odd, then there is a unique solution of the form $c = \frac{1}{2}\left( A + \frac{q-1}{d} \right) + \ell \frac{q-1}{d}$ for some $\ell$. Multiplying this equality by $2$ we obtain
    \begin{align*}
        2c = A + \frac{q-1}{d} + 2 \ell\frac{q-1}{d} = A + \left( 2\ell + 1 \right) \frac{q-1}{d}.
    \end{align*}
    We note that $(2\ell + 1)$ is odd, while $d$ is always even. Therefore $\frac{2\ell + 1}{d}$ will never be an integer, and thus $2c\not \equiv A \pmod{q-1}$. Plugging in this $c$, we obtain
    \begin{align*}
     2\zeta_{\frac{q-1}{d}}^{\mu \left( \frac{1}{2}\left( A + \frac{q-1}{d} \right) + \ell \frac{q-1}{d} \right)} &= 2 \zeta_{\frac{q-1}{d}}^{\mu \frac{1}{2} \left( A + \frac{q-1}{d} \right)} \zeta_{\frac{q-1}{d}}^{\mu \ell \frac{q-1}{d}} = 2 \zeta_{q-1}^{\frac{n}{2} (A + \frac{q-1}{d})} \\
     &= 2\zeta_{q-1}^{A \frac{n}{2}} \zeta_{q-1}^{\frac{q-1}{2} \frac{n}{d}} = 2(-1)^{n/d} a^{n/2}.
\end{align*}

\item If $A$ is even and $\frac{q-1}{d}$ is odd, then there is a unique solution, namely $c \equiv \frac{A}{2} \pmod{\frac{q-1}{d}}$. Any such solution will be an integer of the form $c = \frac{A}{2} + \ell \frac{q-1}{d}$ for some $\ell$, so we may multiply by 2 to obtain
\begin{align*}
    2c = A + 2\ell \frac{q-1}{d}.
\end{align*}
If $d=2$, then this solution yields $2c\equiv A \pmod{q-1}$, so we must omit this value. In this case we see that
\begin{align*}
    D_n(\mathcal{H}_q(a),a) &= \left\{ \zeta_{\frac{q-1}{d}}^{m c} + \zeta_{\frac{q-1}{d}}^{m (A-c)} \colon c\not\equiv \frac{A}{2} \bmod{\frac{q-1}{d}} \right\}_{c=1}^{\frac{q-1}{d}}.
\end{align*}

If $d\ne 2$, then it is not the case that $c$ has to satisfy $2c\equiv A \pmod{q-1}$. This tells us that
\begin{align*}
    D_n(\mathcal{H}_q(a),a) &=  \left\{ \zeta_{\frac{q-1}{d}}^{m c} + \zeta_{\frac{q-1}{d}}^{m (A-c)} \colon c\not\equiv \frac{A}{2} \bmod{\frac{q-1}{d}} \right\}_{c=1}^{\frac{q-1}{d}} \cup \left\{ 2 \zeta_{\frac{q-1}{d}}^{m \frac{A}{2}} \right\}.
\end{align*}
Here we compute that $\zeta_{\frac{q-1}{d}}^{m \frac{A}{2}} = \zeta_{q-1}^{A \frac{n}{2}} = a^{n/2}$.

\item If $A$ is even and $\frac{q-1}{d}$ is even, then there are two solutions, namely $c \equiv \frac{A}{2} \bmod{\frac{q-1}{d}}$ and $c \equiv \frac{A}{2} + \frac{q-1}{2d} \bmod{\frac{q-1}{d}}$. Let's look at these two solutions individually.
\begin{enumerate}
    \item For the case $c\equiv \frac{A}{2}$, we have that $c$ is an integer of the form $c = \frac{A}{2} + \ell \frac{q-1}{d}$. Multiplying by $2$ we obtain $2c = A + 2\ell \frac{q-1}{d}$. If $d=2$, we have that $2c\equiv A \pmod{q-1}$, so this $c$ yields a parabolic element.

    \item For the case $c\equiv \frac{A} {2} + \frac{q-1}{2d}$, we have that $c = \frac{A}{2} + \frac{q-1}{2d} + \ell \frac{q-1}{d}$ for some $\ell$. Multiplying by $2$ yields
    \begin{align*}
        2c = A + (2\ell + 1)\frac{q-1}{d}.
    \end{align*}
    As $2\ell + 1$ is odd and $d$ is even, this choice of $c$ will never satisfy $2c \equiv A \pmod{q-1}$.
\end{enumerate}

\end{enumerate}

\end{proof}

\begin{corollary} The size of the hyperbolic set is:
\begin{align*}
    |D_n(\mathcal{H}_q(a),a)| &= \floor{\frac{q-1}{2d}} + \begin{cases}
    1 & A\cdot \frac{q-1}{d} \text{ odd}; \\
    1 & A \text{ even, and } d\ne 2; \\
    0 & \text{otherwise}. \end{cases}
\end{align*}
\end{corollary}

\begin{lemma}\label{lem:labelname} The hyperbolic sum is
\begin{align*}
    S_q^{\mathcal{H}_q(a)}(D_n(x,a)) &= \begin{cases}
    2 \legendre{a}{q}^{n/d} a^{n/2} & d = q-1; \\
    0 & A\text{ odd, } \frac{q-1}{d} \text{ even}; \\
   (-1)^{n/d} a^{n/2} & A\text{ odd, } \frac{q-1}{d} \text{ odd}; \\
    -a^{n/2} & A\text{ even, } \frac{q-1}{d} \text{ odd, } d=2; \\
    a^{n/2} & A\text{ even, } \frac{q-1}{d} \text{ odd, } d\ne 2; \\
    -2a^{n/2} & A\text{ even, } \frac{q-1}{d} \text{ even, } d=2;    \\
    0 & A\text{ even, } \frac{q-1}{d} \text{ even, } d\ne 2.
   \end{cases}
\end{align*}
\end{lemma}
\begin{proof} We may sum over the hyperbolic image as computed in \autoref{prop:hyperbolic-image} to obtain
\begin{align*}
    S_q^{\mathcal{H}_q(a)}(D_n(x,a)) &= \begin{cases}
    2 \legendre{a}{q}^{n/d} a^{n/2} & d = q-1; \\
    0 & A\text{ odd, } \frac{q-1}{d} \text{ even}; \\
   (-1)^{n/d} a^{n/2} & A\text{ odd, } \frac{q-1}{d} \text{ odd}; \\
    -a^{n/2} & A\text{ even, } \frac{q-1}{d} \text{ odd, } d=2; \\
    a^{n/2} & A\text{ even, } \frac{q-1}{d} \text{ odd, } d\ne 2; \\
    -a^{n/2} +(-1)^{n/d} a^{n/2} & A\text{ even, } \frac{q-1}{d} \text{ even, } d=2 ;   \\
    a^{n/2} + (-1)^{n/d}a^{n/2} & A\text{ even, } \frac{q-1}{d} \text{ even, } d\ne 2.
   \end{cases}
\end{align*}
In the latter two cases, $\frac{n}{d}$ is odd since $\frac{q-1}{d}$ is even, yielding the statement of the proposition.
\end{proof}

A similar analysis can be used to characterize the elliptic image. We observe via \autoref{prop:hyp-ell-para-sets-description} and the Binet formula that the elliptic image is
\allowdisplaybreaks
\begin{equation}\label{eqn:elliptic-image}
\begin{aligned}
    &\quad\quad D_n(\mathcal{E}_p(a),a) = \left\{ \zeta_{q^2-1}^{n(A+k(q-1))}+\zeta_{q^2-1}^{n(Aq-k(q-1))}: 1\le k\le \frac{q+1}{\delta}\ \mathrm{and}\ 2k \not\equiv A \bmod{q+1} \right\}
 \\
    &=\left\{ \zeta_{q^2-1}^{nA} \zeta_{\frac{q+1}{\delta}}^{\mu k} + \zeta_{q^2-1}^{nAq} \zeta_{\frac{q+1}{\delta}}^{-\mu k}\colon 1\le k\le \frac{q+1}{\delta},\ 2k\not\equiv A \pmod{\frac{q+1}{\delta}} \right\} \cup \left\{ 2\zeta_{q^2-1}^{nA} \zeta_{\frac{q+1}{\delta}}^{\mu k} \colon \substack{ 2k\equiv A \bmod{\frac{q+1}{\delta}} \text{ but} \\ 2k\not \equiv A \bmod{q+1}  } \right\}.
\end{aligned}
\end{equation}

Again we may better characterize this in various cases.

\begin{proposition}\label{prop:elliptic-image} We have that the elliptic image $D_n(\mathcal{E}_q(a),a)$ is equal to
\begin{align*}
    \begin{cases}
    \left\{ 2\legendre{a}{q}^{n/\delta} a^{n/2} \right\} & \delta=q+1; \\
    \left\{ \zeta_{q^2-1}^{nA} \zeta_{\frac{q+1}{\delta}}^{\mu k} + \zeta_{q^2-1}^{nAq} \zeta_{\frac{q+1}{\delta}}^{-\mu k} \right\}_{k=1}^{\frac{q+1}{\delta}} & A\text{ odd, } \frac{q+1}{\delta}\text{ even}; \\
    \left\{ \zeta_{q^2-1}^{nA} \zeta_{\frac{q+1}{\delta}}^{\mu k} + \zeta_{q^2-1}^{nAq} \zeta_{\frac{q+1}{\delta}}^{-\mu k}\colon k \ne \frac{1}{2}\left(  A + \frac{q+1}{\delta} \right) \right\}_{k=1}^{\frac{q+1}{\delta}}\cup \left\{ 2(-1)^{n/\delta} a^{n/2}\right\} & A\text{ odd, } \frac{q+1}{\delta}\text{ odd}; \\
    \left\{ \zeta_{q^2-1}^{nA} \zeta_{\frac{q+1}{\delta}}^{\mu k} + \zeta_{q^2-1}^{nAq} \zeta_{\frac{q+1}{\delta}}^{-\mu k}\colon k \ne \frac{A}{2} \right\}_{k=1}^{\frac{q+1}{\delta}} & A\text{ even, } \frac{q+1}{\delta}\text{ odd, } \delta=2; \\
    \left\{ \zeta_{q^2-1}^{nA} \zeta_{\frac{q+1}{\delta}}^{\mu k} + \zeta_{q^2-1}^{nAq} \zeta_{\frac{q+1}{\delta}}^{-\mu k}\colon k\ne \frac{A}{2} \right\}_{k=1}^{\frac{q+1}{\delta}}\cup \left\{ 2a^{n/2} \right\} & A\text{ even, } \frac{q+1}{\delta}\text{ odd, } \delta\ne 2; \\
    \left\{ \zeta_{q^2-1}^{nA} \zeta_{\frac{q+1}{\delta}}^{\mu k} + \zeta_{q^2-1}^{nAq} \zeta_{\frac{q+1}{\delta}}^{-\mu k}\colon k \ne \frac{A}{2},\ \frac{A}{2} + \frac{q+1}{2\delta} \right\}_{k=1}^{\frac{q+1}{\delta}}\cup \left\{ -2a^{n/2} \right\} 
 & A\text{ even, } \frac{q+1}{\delta}\text{ even, } \delta=2; \\
    \left\{ \zeta_{q^2-1}^{nA} \zeta_{\frac{q+1}{\delta}}^{\mu k} + \zeta_{q^2-1}^{nAq} \zeta_{\frac{q+1}{\delta}}^{-\mu k}\colon k \ne \frac{A}{2},\ \frac{A}{2} + \frac{q+1}{2\delta} \right\}_{k=1}^{\frac{q+1}{\delta}}\cup \left\{ 2a^{n/2}, -2a^{n/2} \right\}
 & A\text{ even, } \frac{q+1}{\delta}\text{ even, } \delta\ne 2. \\
    \end{cases}
\end{align*}

\end{proposition}
\begin{proof} We can solve for the congruence $2k\equiv A \bmod{\frac{q+1}{\delta}}$.
\begin{enumerate}
    \item When $\delta = q+1$, we can write $n = \frac{n}{\delta}(q+1)$, from which we see that any element in the elliptic image takes the form
\begin{align*}
    \zeta_{q^2-1}^{n \left( A + k (q-1) \right)} + \zeta_{q^2-1}^{n \left( Aq - k(q-1) \right)} &= 2\zeta_{q^2-1}^{\frac{n}{\delta} A(q+1)} =2a^{n/\delta}.
\end{align*}
We remark that if $\delta = q+1$, we may write
\begin{align*}
    a^{n/2} &= \left( a^{n/\delta} \right)^{\delta/2} = \left( a^{n/\delta} \right)^{\frac{q+1}{2}} = a^{n/\delta} \left( a^{n/\delta} \right)^{\frac{q-1}{2}} = a^{n/\delta} \legendre{a^{n/\delta}}{q}.
\end{align*}
We may verify that $\legendre{a^{n/\delta}}{q} = \legendre{a}{q}^{n/\delta}$, from which we compute
\begin{align*}
    a^{n/\delta} = \legendre{a}{q}^{n/\delta} a^{n/2}.
\end{align*}

    \item In this case there are no solutions to $2k\equiv A \pmod{\frac{q+1}{\delta}}$.
    
    \item In this case, there is a unique solution, namely an integer of the form $k = \frac{1}{2} \left( A + \frac{q+1}{\delta} \right) + \ell \frac{q+1}{\delta}$ for some $\ell$. Multiplying by 2 we get
    \begin{align*}
        2k &= A + \left( 2\ell+1 \right)\frac{q+1}{\delta}.
    \end{align*}
    Since $\delta$ is even, we have that $\frac{2\ell+1}{\delta}$ will never be an integer, so such a $k$ will not satisfy $2k\equiv A \pmod{q+1}$.

    \item There is a unique solution, $k \equiv \frac{A}{2} \pmod{\frac{q+1}{\delta}}$. This will be some integer of the form
\begin{align*}
    k = \frac{A}{2} + \ell \frac{q+1}{\delta}.
\end{align*}
Multiplying by $2$, we have that
\begin{align*}
    2k = A + 2\ell \frac{q+1}{\delta}.
\end{align*}
Thus we have to break into cases based on whether $\delta=2$ or $\delta\ne 2$.
\setcounter{enumi}{4}
\item We see that when $k = \frac{A}{2} + \frac{q+1}{2\delta}$, that the associated element in the elliptic image is
\begin{align*}
    \zeta_{q^2-1}^{n \left( A + \left( \frac{A}{2} + \frac{q+1}{2\delta} \right)(q-1) \right)} + \zeta_{q^2-1}^{n \left( Aq - \left( \frac{A}{2} + \frac{q+1}{2\delta} \right)(q-1) \right)} &= \zeta_{q^2-1}^{n A \frac{q+1}{2}} \left( \zeta_{q^2-1}^{\frac{n}{\delta} \frac{q^2-1}{2}} + \zeta_{q^2-1}^{-\frac{n}{\delta} \frac{q^2-1}{2}}  \right) \\
    &= a^{n/2} \left( (-1)^{n/\delta} + (-1)^{-n/\delta} \right).
\end{align*}
Since $\frac{q+1}{\delta}$ is even, we have that $n/\delta$ is odd, so the above reduces to $-2a^{n/2}$.
\end{enumerate}
\end{proof}

\begin{corollary} The size of the elliptic set is:
\begin{align*}
    |D_n(\mathcal{E}_q(a),a)| &= \floor{\frac{q+1}{2\delta}} + \begin{cases}
    1 & A\cdot \frac{q+1}{\delta} \text{ odd}; \\
    1 & A \text{ even, and } \delta\ne 2; \\
    0 & \text{otherwise}.
    \end{cases}
\end{align*}
\end{corollary}

\begin{lemma}\label{lem:elliptic-sum} The elliptic sum is
\begin{align*}
        S_q^{\mathcal{E}_q(a)}(D_n(x,a)) &= \begin{cases}
        2\legendre{a}{q}^{n/\delta} a^{n/2} & \delta = q+1; \\
    0
    & A\text{ odd, } \frac{q+1}{\delta}\text{ even}; \\
   (-1)^{n/\delta}a^{n/2}
    & A\text{ odd, } \frac{q+1}{\delta}\text{ odd}; \\
    - a^{n/2}
    & A\text{ even, } \frac{q+1}{\delta}\text{ odd, } \delta=2; \\
    a^{n/2}
    & A\text{ even, } \frac{q+1}{\delta}\text{ odd, } \delta\ne 2; \\
    -2 a^{n/2}
 & A\text{ even, } \frac{q+1}{\delta}\text{ even, } \delta=2; \\
    0
 & A\text{ even, } \frac{q+1}{\delta}\text{ even, } \delta\ne 2. \\
    \end{cases}
\end{align*}
\end{lemma}

\begin{proof} We may sum over the elliptic image to get
\begin{align*}
    S_q^{\mathcal{E}_q(a)}(D_n(x,a)) &= \begin{cases}
    2\legendre{a}{q}^{n/\delta} a^{n/2} & \delta = q+1; \\
    0
    & A\text{ odd, } \frac{q+1}{\delta}\text{ even}; \\
   (-1)^{n/\delta}a^{n/2}
    & A\text{ odd, } \frac{q+1}{\delta}\text{ odd}; \\
    - a^{n/2}
    & A\text{ even, } \frac{q+1}{\delta}\text{ odd, } \delta=2; \\
    a^{n/2}
    & A\text{ even, } \frac{q+1}{\delta}\text{ odd, } \delta\ne 2; \\
    - a^{n/2} + (-1)^{n/\delta}a^{n/2}
 & A\text{ even, } \frac{q+1}{\delta}\text{ even, } \delta=2; \\
    (-1)^{n/\delta}a^{n/2} + a^{n/2}
 & A\text{ even, } \frac{q+1}{\delta}\text{ even, } \delta\ne 2. \\
    \end{cases}
\end{align*}
In the latter two cases, since $\frac{q+1}{\delta}$ is even and coprime to $n/\delta$, we have that $n/\delta$ is odd, which gives the statement of the lemma.
\end{proof}

In order to characterize potential overlaps in the images of the hyperbolic, elliptic, and parabolic sets, it will be easier to break into the case of $a$ a residue and non-residue.

\subsection{Overlaps in the non-residue case}
Via \autoref{prop:possible-hyp-ellip-overlap}, we know that any overlap in the hyperbolic and elliptic images must be a subset of $\left\{ \pm 2a^{n/2} \right\}$, and similarly we know the parabolic image to be $\left\{ 2a^{n/2} \right\}$ when $a$ is a residue, and empty otherwise. So it suffices to determine when either of $\pm 2a^{n/2}$ lie in the hyperbolic and elliptic image.

We remark however that these images will occur in the hyperbolic image precisely when the second set in \autoref{eqn:hyperbolic-image} is nonempty, therefore we can understand these images via the work already done in \autoref{prop:hyperbolic-image}. Similarly for the elliptic case we have solved for when $\pm 2a^{n/2}$ lies in the elliptic image in \autoref{prop:elliptic-image}. We can summarize these findings as follows.

\begin{proposition}\label{prop:nonresidue-EH-overlap} If $a$ is a non-residue, then the elliptic and hyperbolic overlap is the following:
\begin{align*}
    D_n(\mathcal{H}_q(a),a) \cap D_n(\mathcal{E}_q(a),a) &= \begin{cases} \left\{ 2a^{n/2} \right\} & 2^r \mid n; \\
    \emptyset & \text{otherwise.} \end{cases}
\end{align*}
\end{proposition}
\begin{proof} Via \autoref{prop:hyperbolic-image} when $A$ is odd, we will have $2(-1)^{n/d}a^{n/2}$ in the hyperbolic image if $\frac{q-1}{d}$ is odd (including when $\frac{q-1}{d} = 1$). Similarly via \autoref{prop:elliptic-image}, we will have $2(-1)^{n/\delta}a^{n/2}$ in the elliptic image when $\frac{q+1}{\delta}$ is odd (including when $\frac{q+1}{\delta} = 1$). Therefore in order to have any overlap we must have that the parities of $\frac{n}{d}$ and $\frac{n}{\delta}$ coincide.  Since both $\frac{q-1}{d}$ and $\frac{q+1}{\delta}$ are odd, via \autoref{prop:parity-constraints} in order for the parities of $\frac{n}{d}$ and $\frac{n}{\delta}$ to agree, they must both be even. This condition is equivalent to $2^r\mid n$ by \autoref{prop:r-d-delta-relationship}.
\end{proof}

In the case where $A$ is odd, we can decompose the hyperbolic sum by making the $d=q-1$ case a separate condition as follows:
\allowdisplaybreaks
\begin{align*}
    S_q^{\mathcal{H}_q(a)}(D_n(x,a)) &= \begin{cases}
    2(-1)^{n/d}a^{n/2} & d = q-1; \\
    0 & \frac{q-1}{d} \text{ even}; \\
   (-1)^{n/d} a^{n/2} & \frac{q-1}{d} \text{ odd,}\ne 1, \\
   \end{cases} \\
   &= \begin{cases} (-1)^{n/d}a^{n/d} & \frac{q-1}{d}\text{ odd}; \\ 0 & \frac{q-1}{d}\text{ even}, \end{cases} + \begin{cases} (-1)^{n/d}a^{n/2} & d=q-1; \\ 0 & \text{else.} \end{cases}
\end{align*}
A similar argument shows that
\begin{align*}
    S_q^{\mathcal{E}_q(a)}(D_n(x,a)) &= \begin{cases} (-1)^{n/\delta}a^{n/2} & \frac{q+1}{\delta}\text{ odd}; \\
    0 & \frac{q+1}{\delta}\text{ even}, \end{cases} + \begin{cases} 
    (-1)^{n/\delta} a^{n/2} & \delta=q+1; \\
    0 & \text{else}. \end{cases}
\end{align*}

\begin{lemma}\label{lem:A-odd-sum} Let $A$ be odd. Then the sum is given as
\begin{align*}
    S_q(D_n(x,a)) &= \begin{cases}
    0 & 2^{r-1}\mid n; \\
    -(-1)^{\frac{n}{d}+\frac{n}{\delta}} a^{n/2} & 2^{r-1}\nmid n, \end{cases} + \begin{cases}
    (-1)^{n/d}a^{n/2} & d=q-1; \\
    0 & \text{else}, \end{cases} + \begin{cases}
    (-1)^{n/\delta}a^{n/2} & \delta=q+1; \\
    0 & \text{else}. \end{cases}
\end{align*}
\end{lemma}
\begin{proof} We can combine the three conditions:
\begin{align*}
    \begin{cases} (-1)^{n/d}a^{n/d} & \frac{q-1}{d}\text{ odd}; \\ 
    0 & \frac{q-1}{d}\text{ even}, \end{cases} + \begin{cases} 
    (-1)^{n/\delta}a^{n/2} & \frac{q+1}{\delta}\text{ odd}; \\ 
    0 & \frac{q+1}{\delta}\text{ even}, \end{cases} - \begin{cases} 
    2a^{n/2} & \frac{n}{d},\ \frac{n}{\delta}\text{ even}; \\
    0 & \text{else.} \end{cases}
\end{align*}
Rewriting these conditions using $h$, $\ell$, and $r = h+\ell$, we have
\begin{align*}
    \begin{cases} (-1)^{n/d}a^{n/d} & 2^h\mid n; \\ 
    0 & 2^h\nmid n,
    \end{cases} + \begin{cases}
    (-1)^{n/\delta}a^{n/2} & 2^\ell \mid n; \\ 
    0 & 2^\ell \nmid n, 
    \end{cases} - \begin{cases}
    2a^{n/2} & 2^{h+\ell}\mid n; \\
    0 & 2^{h+\ell}\nmid n. \end{cases}
\end{align*}
Combining conditions, we can see that this simplifies to
\begin{align*}
    \begin{cases}
    0 & 2^r\mid n; \\
    (-1)^{n/d}a^{n/2} + (-1)^{n/\delta}a^{n/2} & 2^{r-1}\mid\mid n; \\
    (-1)^{n/d}a^{n/2}& 2^{r-1}\nmid n,\ \ell > h; \\
    (-1)^{n/\delta}a^{n/2} & 2^{r-1}\nmid n,\ h>\ell.
    \end{cases}
\end{align*}
Via \autoref{prop:parity-constraints}, when $2^{r-1}\mid\mid n$, we have that $\frac{n}{d}$ and $\frac{n}{\delta}$ have opposite parities, so we can merge the $2^{r-1}\mid\mid n$ and $2^r\mid n$ conditions.

When $\ell > h$, we have that $\ell= r-1$, so $2^{r-1}\nmid n$ is equivalent to $\frac{q+1}{\delta}$ being even, which implies $\frac{n}{\delta}$ is odd. Similarly in the last case, $\frac{n}{d}$ is odd, so we can merge the last two conditions. This gives
\begin{align*}
    \begin{cases}
    0 & 2^{r-1}\mid n; \\
    -(-1)^{\frac{n}{d}+\frac{n}{\delta}} a^{n/2} & 2^{r-1}\nmid n. \end{cases}
\end{align*}
Adding back the $d=q-1$ and $\delta=q+1$ conditions, we obtain the desired statement.
\end{proof}

\subsection{Overlaps: the residue case}
In the residue setting, the parabolic image is precisely $\{2a^{n/2}\}$. Thus we need to study when $2a^{n/2}$ can lie in the hyperbolic and elliptic images in order to understand when they admit overlap with the parabolic image.

We can begin with the hyperbolic-elliptic overlap. We note that in almost all the cases in which they have overlap, the number at which they overlap is $2a^{n/2}$, which is parabolic. In particular, when this occurs, we will have 
\[
S_q^{\mathcal{H}_q(a),\mathcal{E}_q(a)}(D_n(x,a)) - S_q^{\mathcal{H}_q(a),\mathcal{E}_q(a),\mathcal{P}_q(a)}(D_n(x,a)) = 0.
\]
We note that this difference will only ever be nonzero when $-2a^{n/2}$ lies in both the hyperbolic and elliptic images. However this cannot occur.
\begin{proposition}\label{prop:HE-HEP-difference} Let $a$ be a residue. Then
\begin{align*}
    S_q^{\mathcal{H}_q(a),\mathcal{E}_q(a)}(D_n(x,a)) - S_q^{\mathcal{H}_q(a),\mathcal{E}_q(a),\mathcal{P}_q(a)}(D_n(x,a)) = 0.
\end{align*}
\end{proposition}
\begin{proof} We see that in the residue case, $\legendre{a}{q}=1$. So $-2a^{n/2}$ can only lie in the hyperbolic image when $\frac{q-1}{d}$ even, while $-2a^{n/2}$ can only lie in the elliptic image when $\frac{q+1}{\delta}$ is even. However these both can't simultaneously occur by \autoref{prop:parity-constraints}. Thus any hyperbolic-elliptic overlap occurs at $2a^{n/2}$ which therefore is parabolic as well.
\end{proof}

We can easily characterize the hyperbolic-parabolic and elliptic-parabolic overlaps by observing when $2a^{n/2}$ lies in the hyperbolic and elliptic images.

\begin{proposition}\label{prop:HP-overlap} Let $a$ be a residue. Then
\begin{align*}
    S_q^{\mathcal{H}_q(a),\mathcal{P}_q(a)}(D_n(x,a)) &=%
    \begin{cases}
    2a^{n/2} & d\ne 2; \\
    0 & \text{otherwise}.
    \end{cases}
\end{align*}
\end{proposition}

\begin{proposition}\label{prop:EP-overlap} Let $a$ be a residue. Then
\begin{align*}
    S_q^{\mathcal{E}_q(a),\mathcal{P}_q(a)}(D_n(x,a)) &=%
    \begin{cases}
    2a^{n/2} & \delta\ne 2; \\
    0 & \text{otherwise}.
    \end{cases}
\end{align*}
\end{proposition}

Now we can characterize the entire sum by combining these sums and their overlaps. Before doing so, we can begin to cancel some of the sums with others. First we can combine the elliptic sum with the elliptic-parabolic overlap:
\allowdisplaybreaks
\begin{align*}
    S_q^{\mathcal{E}_q(a)}(D_n(x,a)) - S_q^{\mathcal{E}_q(a),\mathcal{P}_q(a)}(D_n(x,a)) &= \begin{cases}
        2a^{n/2} & \delta = q+1; \\
    - a^{n/2} & \frac{q+1}{\delta}\text{ odd, } \delta=2; \\
    a^{n/2} & \frac{q+1}{\delta}\text{ odd, } \delta\ne 2; \\
    -2 a^{n/2} & \frac{q+1}{\delta}\text{ even, } \delta=2; \\
    0 & \frac{q+1}{\delta}\text{ even, } \delta\ne 2, \\
    \end{cases} - \begin{cases}
    2a^{n/2} & \delta\ne 2; \\
    0 & \text{otherwise},
    \end{cases} \\
    &= \begin{cases}
    0 & \delta = q+1; \\
    -a^{n/2} & \frac{q+1}{\delta}\text{ odd, }\ne 1; \\
    -2a^{n/2} & \frac{q+1}{\delta}\text{ even},
    \end{cases} \\
    &=  \begin{cases}
    -a^{n/2} & \frac{q+1}{\delta}\text{ odd}; \\ 
    -2a^{n/2}& \frac{q+1}{\delta}\text{ even}, \end{cases} + \begin{cases}
    a^{n/2} & \delta=q+1; \\ 0 & \text{else}. \end{cases}
\end{align*}
Combining the hyperbolic and the hyperbolic-parabolic overlap we see
\allowdisplaybreaks
\begin{align*}
    S_q^{\mathcal{H}_q(a)}(D_n(x,a)) - S_q^{\mathcal{H}_q,(a)\mathcal{P}_q(a)}(D_n(x,a)) &=
    \begin{cases}
    2a^{n/2} & d = q-1; \\
    -a^{n/2} & \frac{q-1}{d} \text{ odd, } d=2; \\
    a^{n/2} & \frac{q-1}{d} \text{ odd, } d\ne 2; \\
    -2a^{n/2} & \frac{q-1}{d} \text{ even, } d=2;    \\
    0 & \frac{q-1}{d} \text{ even, } d\ne 2,
   \end{cases} - \begin{cases}
    2a^{n/2} & d\ne 2; \\
    0 & \text{otherwise},
    \end{cases} \\
    &= \begin{cases}
    0 & d=q-1; \\
    -a^{n/2} & \frac{q-1}{d}\text{ odd, }\ne 1; \\
    -2a^{n/2} & \frac{q-1}{d}\text{ even},
    \end{cases} \\
    &= \begin{cases}
    -a^{n/2} & \frac{q-1}{d}\text{ odd}; \\
    -2a^{n/2} & \frac{q-1}{d}\text{ even},
    \end{cases} + \begin{cases}
    a^{n/2} & d=q-1; \\
    0 & \text{else}. \end{cases}
\end{align*}

\begin{lemma}\label{lem:A-even-sum} Let $a$ be a residue. Then
\begin{align*}
    S_q(D_n(x,a)) &= \begin{cases}
    0 & 2^{r-1}\mid n; \\
    -a^{n/2} & 2^{r-1}\nmid n,
    \end{cases} + \begin{cases}
    a^{n/2} & \delta=q+1; \\
    0 & \text{else},
    \end{cases} + \begin{cases}
    a^{n/2} & d=q-1; \\
    0 & \text{else}. \end{cases}
\end{align*}
\end{lemma}

\begin{proof} We first combine the conditions
\begin{align*}
    \begin{cases}
    -a^{n/2} & \frac{q+1}{\delta}\text{ odd}; \\ 
    -2a^{n/2}& \frac{q+1}{\delta}\text{ even},
    \end{cases} + \begin{cases}
    -a^{n/2} & \frac{q-1}{d}\text{ odd}; \\
    -2a^{n/2} & \frac{q-1}{d}\text{ even},
    \end{cases} = \begin{cases}
    -2a^{n/2} & 2^{r-1}\mid n; \\
    -3a^{n/2} & \text{else}.
    \end{cases}
\end{align*}
Adding back $2a^{n/2}$ from the parabolic sum and the $d=q-1$ and $\delta = q+1$ conditions yields the statement of the theorem.
\end{proof}

Combining \autoref{lem:A-odd-sum} and \autoref{lem:A-even-sum} yields the main theorem of the paper.

\subsection{Examples}
A priori, in order to totally characterize the sum of a family of Dickson polynomials, we must understand the quadratic character of $a$, whether $(q-1)$, $(q+1)$, or $2^{r-1}$ divide $n$, and the parities of $\frac{n}{d}$ and $\frac{n}{\delta}$. In the case of the Lucas polynomials $L_n(x) = D_n(x,a)$, many of these conditions coalesce --- for example the quadratic character of $-1$ is dependent upon the residue of our prime modulo four, which also determines possible parities of $\frac{n}{d}$ and $\frac{n}{\delta}$. In fact, modulo a fixed prime, knowledge of $d$ and $\delta$ alone determines the residue sum.
\begin{example} As in Figure~\ref{tab:table1}, consider when $p=7$. In this case there are a very limited number of possibilities for $d$ and for $\delta$. Since the values $d$ and $\delta$ completely determine $S_7(L_n)$, we provide the following table.
\begin{figure}[H]\label{fig:s7ln}
\caption{Possible values for $S_7(L_n)$.}
\par\noindent\rule{0.4\textwidth}{0.4pt} \\
\vspace{0.2em}
\begin{tabular}{c | c | c}
$d$ & $\delta$ & $S_7(L_n)$ \\
\midrule
 2 & 2 & 1\\
 2 & 4 & 1\\
 2 & 8 & -1
 \end{tabular} \quad
 \begin{tabular}{c | c | c}
$d$ & $\delta$ & $S_7(L_n)$ \\
\midrule
 6 & 2 & 2\\
 6 & 4 & 2\\
 6 & 8 & 0
 \end{tabular}
\centering
\end{figure}
To provide an example when $p \equiv 1 \pmod{4}$, we can write an analogous table for $p=29$, although as expected it is much larger. Possible even values for $d = \gcd(n,28)$ are $d\in \left\{2,4,14,28\right\}$, while $\delta = \gcd(n,60)$ must be even as well, and thus lies in $\delta \in \left\{2,4,6,10,12,20,30,60\right\}$. However we remark that as 28 and 60 are both divisible by 4, we have that $4\mid d$ if and only if $4\mid \delta$, which gives us a restriction on the possible pairs that can show up. This yields the following table.
\begin{figure}[H]\label{fig:s29ln}
\caption{Possible values for $S_{29}(L_n)$.}
\par\noindent\rule{0.81\textwidth}{0.4pt} \\
\vspace{0.2em}
\begin{tabular}{c | c | c}
$d$ & $\delta$ & $S_{29}(L_n)$ \\
\midrule
 2 & 2 & 1\\
 2 & 6 & 1 \\
 2 & 10 & 1 \\
 2 & 30 & 0 \\
 \end{tabular} \quad
\begin{tabular}{c | c | c}
$d$ & $\delta$ & $S_{29}(L_n)$ \\
\midrule
 4 & 4 & 0\\
 4 & 12 & 0 \\
 4 & 20 & 0\\
 \multicolumn{3}{c}{\textcolor{white}{.}}
 \end{tabular} \quad
 \begin{tabular}{c | c | c}
$d$ & $\delta$ & $S_{29}(L_n)$ \\
\midrule
 14 & 2 & 1\\
 14 & 6 & 1 \\
 14 & 10 & 1 \\
 14 & 30 & 0 \\
 \end{tabular} \quad
 \begin{tabular}{c | c | c}
$d$ & $\delta$ & $S_{29}(L_n)$ \\
\midrule
 28 & 4 & 1\\
 28 & 12 & 1 \\
 28 & 20 & 1 \\
  \multicolumn{3}{c}{\textcolor{white}{.}}
 \end{tabular} \quad
\centering
\end{figure}
\end{example}

\begin{example} As another example, let $T_n(x) = \cos(n\arccos(x))$ denote the $n$th Chebyshev polynomial. It is well known that these are related to the Dickson polynomials for $a=1$ via the equality
\[
D_n(2x,1) = 2T_n(x).
\]
In particular this implies that for odd characteristic we have
\begin{align*}
    S_q(T_n(x)) &= \frac{1}{2}S_q(D_n(x,1)).
\end{align*}
Invoking \autoref{thm:S-sum}, we may provide a characterization of this sum. As in the Lucas case, it admits an extremely constrained number of possible values. We may verify for all $n$ and $q$ that
\[
S_q(T_n(x)) \in \left\{\pm \frac{1}{2}, 0, 1\right\}.
\]
\end{example}

\section{Further Directions and Conclusion}

One natural direction to follow is to find other two step recurrences for which the above techniques can be employed. If one defines the polynomials $P_n(x)$ recursively by 
    \begin{align*}
        P_n(x)=Ax\cdot P_{n-1}(x)+B\cdot P_{n-2}(x),
    \end{align*}
given initial conditions
\begin{align*}
    P_0(x)&=C \text{ and } P_1(x) = \frac{AC}{2}x,
\end{align*}
where $A,B \in \Z$ and $C$ is an even integer, then $P_n(x)$ shares many of the same properties with the Dickson polynomials $D_n(x,a)$. In particular, $P_n(x)$ is of degree $n$ for each $n$, is odd for $n$ odd and even for $n$ even, and admits the following Binet formula expansion:
    \begin{align*}
        &\qquad \quad P_n(x)=\frac{C}{2}
        \left(\alpha(x)^n + \beta(x)^n\right),
    \end{align*}
where $\alpha(x) = \frac{Ax+\sqrt{(Ax)^2+4B}}{2}$ and $\beta(x) = \frac{Ax-\sqrt{(Ax)^2+4B}}{2}$.
By studying the quadratic character of $A^2x^2+4B$, we obtain sets akin to the hyperbolic, elliptic, and parabolic from above. If we set $A=B=1$, then $P_n(x)=\frac{C}{2}L_n(x)$ and the values for $S_q(P_n)$ are in the set $\left\{\frac{-C}{2},0,\frac{C}{2},C\right\}$.

Another family of interest would be the Fibonacci polynomials, given by the initial conditions $F_1(x) = 1$, $F_2(x) = x$, and the recurrence relation
\[
    F_n(x) = x\cdot F_{n-1}(x) + F_{n-2}(x).
\]
Due to the discrepancy between the indexing conventions on Dickson polynomials versus Fibonacci polynomials, the Fibonacci polynomials of even degree will always be odd, and hence $S_p(F_{2n}) = 0$ for all $p$ and $n$. An investigation of $S_p(F_n)$ at the prime 7 when $n$ is odd displays that sums over residues of Fibonacci polynomials are far less constrained than their Lucas counterparts.
\begin{figure}[H]\label{table3}
\caption{Investigation of $S_7(F_{2n-1})$ for $1\le n\le 40$.}
\par\noindent\rule{0.7\textwidth}{0.4pt} \\
\vspace{0.2em}
\begin{tabular}{c | c}
$n$ & $S_7(F_n)$ \\
\midrule
 1 & 1 \\
 3 & 4 \\
 5 & 3 \\
 7 & 0 \\
 9 & 5 \\
 11 & 3 \\
 13 & 6 \\
 15 & 6
 \end{tabular}\quad
 \begin{tabular}{c | c}
$n$ & $S_7(F_n)$ \\
\midrule
 17 & 1 \\
 19 & 6 \\
 21 & 3 \\
 23 & 0 \\
 25 & 0 \\
 27 & 3 \\
 29 & 6 \\
 31 & 1 
  \end{tabular}\quad
 \begin{tabular}{c | c}
$n$ & $S_7(F_n)$ \\
\midrule
 33 & 6 \\
 35 & 6 \\
 37 & 3 \\
 39 & 5 \\
 41 & 0 \\
 43 & 3 \\
 45 & 4 \\
 47 & 1
 \end{tabular}\quad
 \begin{tabular}{c | c}
$n$ & $S_7(F_n)$ \\
\midrule
 49 & 1 \\
 51 & 4 \\
 53 & 3 \\
 55 & 0 \\
 57 & 5 \\
 59 & 3 \\
 61 & 6 \\
 63 & 6
\end{tabular}\quad \begin{tabular}{c | c}
$n$ & $S_7(F_n)$ \\
\midrule
 65 & 1 \\
 67 & 6 \\
 69 & 3 \\
 71 & 0 \\
 73 & 0 \\
 75 & 3 \\
 77 & 6 \\
 79 & 1
\end{tabular}
\centering
\end{figure}

An interesting direction of research would be to classify these sums in an analogous procedure to that presented in this paper, and begin to characterize the size of the image sets of Fibonacci polynomials modulo $p$. We observe that there is some $(p^2-1)$-fold periodicity in this table which is analogous to that observed for the Dickson polynomials. These values are also palindromic about $\frac{p^2-1}{2}$, which in the Dickson polynomials is explained by replacing $n$ by $p^2-1-n$ in Theorem~\ref{thm:S-sum}.

\color{black}

\bibliographystyle{amsalpha}
\bibliography{citations.bib}{}

\providecommand{\bysame}{\leavevmode\hbox to3em{\hrulefill}\thinspace}
\providecommand{\MR}{\relax\ifhmode\unskip\space\fi MR }
\providecommand{\MRhref}[2]{%
  \href{http://www.ams.org/mathscinet-getitem?mr=#1}{#2}
}
\providecommand{\href}[2]{#2}
\begin{thebibliography}{CGCM88}

\bibitem[BGS16]{bgs}
Jean Bourgain, Alexander Gamburd, and Peter Sarnak, \emph{Markoff surfaces and
  strong approximation: 1}, (2016), arXiv 1607.01530.

\bibitem[BSD59]{BSD}
B.~J. Birch and H.~P.~F. Swinnerton-Dyer, \emph{Note on a problem of {C}howla},
  Acta Arith. \textbf{5} (1959), 417--423 (1959). \MR{113844}

\bibitem[CGCM88]{CGM}
Wun~Seng Chou, Javier Gomez-Calderon, and Gary~L. Mullen, \emph{Value sets of
  {D}ickson polynomials over finite fields}, J. Number Theory \textbf{30}
  (1988), no.~3, 334--344. \MR{966096}

\bibitem[Coh73]{cohen}
S.~D. Cohen, \emph{The values of a polynomial over a finite field}, Glasgow
  Math. J. \textbf{14} (1973), 205--208. \MR{347781}

\bibitem[Cus98]{cusick}
Thomas~W. Cusick, \emph{Value sets of some polynomials over finite fields
  {${\rm GF}(2^{2m})$}}, SIAM J. Comput. \textbf{27} (1998), no.~1, 120--131.
  \MR{1614876}

\bibitem[FSHW20]{FHW}
Carrie Fincher-Smith, Joshua Harrington, and Tony~W.H. Wong, \emph{Sums of
  distinct polynomial residues}, Preprint.

\bibitem[Gau66]{Gauss}
Carl~Friedrich Gauss, \emph{Disquisitiones arithmeticae}, Translated into
  English by Arthur A. Clarke, S. J, Yale University Press, New Haven,
  Conn.-London, 1966. \MR{0197380}

\bibitem[GHM17]{GHM}
Samuel~S. Gross, Joshua Harrington, and Laurel Minott, \emph{Sums of polynomial
  residues}, Irish Math. Soc. Bull. (2017), no.~79, 31--37. \MR{3701183}

\bibitem[Kan15]{kantor}
Richard Kantor, \emph{\"{U}ber die {A}nzahl inkongruenter {W}erte ganzer,
  rationaler {F}unktionen}, Monatsh. Math. Phys. \textbf{26} (1915), no.~1,
  24--39. \MR{1548638}

\bibitem[KK90]{knopfmacher}
Arnold Knopfmacher and John Knopfmacher, \emph{The distribution of values of
  polynomials over a finite field}, Linear Algebra Appl. \textbf{134} (1990),
  145--151. \MR{1060017}

\bibitem[Ste04]{Stetson}
Orlando~S. Stetson, \emph{Triangular {R}esidues}, Amer. Math. Monthly
  \textbf{11} (1904), no.~5, 106--107. \MR{1516128}

\bibitem[Sun06]{sun}
Zhi-Hong Sun, \emph{{On the number of incongruent residues of $x^4 + ax^2 + bx$
  modulo $p$}.}, Journal of Number Theory \textbf{119} (2006), 210--241.

\bibitem[Uch54]{uchiyama}
Sabur\^{o} Uchiyama, \emph{Sur le nombre des valeurs distinctes d'un
  polyn\^{o}me \`a coefficients dans un corps fini}, Proc. Japan Acad.
  \textbf{30} (1954), 930--933. \MR{68581}

\bibitem[Vol89]{voloch}
J.~F. Voloch, \emph{On the number of values taken by a polynomial over a finite
  field}, Acta Arith. \textbf{52} (1989), no.~2, 197--201. \MR{1005605}

\bibitem[vS08]{vonsterneck}
R.~Daublebsky von Sterneck, \emph{{\"{U}ber die Anzahl inkongruenter Werte, die
  eine ganze Funktion dritten Grades annimmt}.}, Sitzungsber. Akad. Wiss. Wien
  (2A) \textbf{114} (1908), 711--717.

\end{thebibliography}
\Addresses

\end{document}